\documentclass[11pt,lettersize,leqno,a4paper]{amsart}
\usepackage[utf8]{inputenc}
\usepackage{amsfonts,amssymb,amscd,amsmath,amsthm}
\usepackage{graphicx}
\usepackage{xypic}

\newcommand{\C}{\mathbb{C}}
\newcommand{\Z}{\mathbb{Z}}
\newcommand{\R}{\mathbb{R}}
\newcommand{\ep}{\epsilon}

\newtheorem{theorem}{Theorem}[section]

\newtheorem{proposition}[theorem]{Proposition}

\newtheorem{rem}[theorem]{Remark}
\newtheorem{exam}[theorem]{Example}

\title{Topology of certain quotient spaces of Stiefel manifolds}
\author{Samik Basu}
\email{samik.basu2@gmail.com; samik@rkmvu.ac.in}
\address{Department of Mathematics,
 Vivekananda University,
 Belur, Howrah 711202,
West Bengal, India.}

\author{B. Subhash}
\email{subhash02@gmail.com;subhash@iisertirupati.ac.in}
\address{Indian Institute of Science Education and research Tirupati  
Mangalam, Tirupati -517507. 
Andra Pradesh, INDIA }

\begin{document}

\maketitle

\begin{abstract}
We compute the cohomology of the right generalised projective Stiefel manifolds and use it to find bounds on the rank of the complementary bundle for certain vector bundles. Further the cohomology computations are also used to find bounds on the span and on the immersibility of the manifolds in certain cases.
\end{abstract}

\section{Introduction}
The study of Stiefel manifolds and their quotients have a long history. Their topology have played a fundamental role in solving many problems such as the celebrated solution of the vector field problem on the sphere by Adams. The cohomology of projective Stiefel manifolds over the field of real numbers was computed in \cite{bau-brow} and \cite{git-han}. These were used to prove immersion results for real projective spaces in \cite{git}.

In the complex case, projective Stiefel manifolds were studied in \cite{coh}, their cohomology was computed and a universal property was associated to these manifolds. As consequences the authors conclude non existence of certain sections to appropriate bundles over projective spaces and lens spaces. Following this the question of parallelizability of projective Stiefel manifolds was considered in \cite{par}.

This paper deals with right generalised projective Stiefel manifolds which were studied in \cite{Gon-Su}. These are obtained as quotients of Stiefel manifolds $W_{n,k}$  (the space of $k$ orthonormal frames in $\C^n$) under an action of the circle group $S^1$. The action is given by the formula $z\cdot (v_1,\cdots, v_k)\mapsto (z^{l_1}v_1, \cdots,z^{l_k}v_k)$ for a fixed primitive tuple $(l_1,\cdots, l_k)\in \Z^k$. The corresponding quotient space is called a right generalised projective Stiefel manifold $P_\ell W_{n,k}$. This can be realised as the homogeneous space $U(n)/S^1\times U(n-k)$.

In this paper we are motivated by the computations for projective Stiefel manifolds and attempt to search similar relations for the right generalised projective Stiefel manifolds. We compute the cohomology of these manifolds (cf. Theorem \ref{cohp}) and observe that they carry a universal property for complementary bundles of a certain sum of line bundles (cf. Theorem \ref{univ}). Using these theorems we deduce some implications for line bundles over projective and lens spaces. 

The vector field problem on Stiefel manifolds and their quotients has been a subject of considerable interest. For the projective Stiefel manifold over $\R$ a number of results and conjectures are mentioned in \cite{kor-zv}. We approach this question for right generalised Stiefel manifolds by computing Pontrjagin classes. A similar argument also yields results on non-immersibility in certain Euclidean spaces (cf. Theorem \ref{span}, Theorem \ref{imm}). Our methods improve the results in \cite{Gon-Su} where the  parallelizability question is settled and the span is related to the stable span.   

The paper is organised as follows. In section 2, we describe the cohomology of the manifolds $P_\ell W_{n,k}$ and introduce an universal property of these spaces. In section 3, we consider applications to the rank of complimentary bundles. In section 4, we obtain applications to the vector field problem and the immersion problem in the case that $k=2.$

\section{Some cohomology computations}\label{coh}
In this section we compute the cohomology of the manifolds $P_\ell W_{n,k}$. Our method involves an interplay between Serre spectral sequences of various fibrations and enables us to conclude a universal property for $P_\ell W_{n,k}$. Throughout we assume that the gcd of $(l_1,\cdots ,l_k)$ is $1$.

Recall that the cohomology of the unitary group $U(n)$ is an exterior algebra in generators in degrees $1,3,\cdots, 2n-1$. We denote this expression by  $H^*(U(n)) = \Lambda(y_{1}, \cdots, y_n)$ where the class $y_j$ lies in the degree $2j-1$. The Stiefel manifold $W_{n,k}$ is homeomorphic to $U(n)/U(n-k)$ and its cohomology is given by $\Lambda (y_{n-k+1},\cdots, y_n)$. 

Recall that the principal $S^1$ fibration $W_{n,k} \to P_\ell W_{n,k}$ yields a fibration $W_{n,k}\to P_\ell W_{n,k} \to BS^1$ the latter space being $\C P^\infty$. Note that the Stiefel manifold $W_{n,k}$ also fibres over the flag manifold $F(1,\cdots,1,n-k)$ of sequences of flags of $k$ subspaces with dimensions increasing by $1$ in each step. This fits into a principal $(S^1)^k$ fibration $W_{n,k}\to F(1,\ldots,1,k)$. The $S^1$ action whose orbits are the manifold $P_\ell W_{n,k}$ comes from the inclusion of $S^1$ in $(S^1)^k$ given by $\Phi_\ell: z\mapsto (z^{l_1},\cdots, z^{l_k})$. This induces a commutative sequence of fibrations 
$$\xymatrix{ S^1 \ar[r]^{\Phi_\ell} \ar[d]    & (S^1)^k \ar[d] \\
                       W_{n,k} \ar@{=}[r] \ar[d]         & W_{n,k} \ar[d] \\
                     P_{\ell}W_{n,k} \ar[r] \ar[d]  & F(1,\cdots,1,n-k) \ar[d]\\
                     \C P^{\infty} \ar[r]^{\phi_\ell} & (\C P^{\infty})^k }$$
These fibre sequences extend one step further. Consider $Gr_k(\C^n)$ the Grassmann manifold of $k$ planes in $\C^n$ which is the quotient $U(n)/U(k)\times U(n-k)$. One has a principal $U(k)$ bundle $W_{n,k} \to Gr_k(\C^n)$ and the map $(S^1)^k\to U(k)$ given by the inclusion of diagonal matrices forms a similar diagram of principal fibrations as above. Putting all this together one obtains a commutative diagram 
$$\xymatrix{S^1 \ar[r]\ar[d] & (S^1)^k \ar[r] \ar[d] & U(k) \ar[d] \\ 
                 W_{n,k}\ar@{=}[r] \ar[d]  & W_{n,k}\ar@{=}[r] \ar[d]  & W_{n,k}\ar[d]\\
 P_{\ell}W_{n,k} \ar[r] \ar[d] & F(1,\cdots,1,n-k) \ar[r] \ar[d] & G_k(\C^n)\ar[d]\\
  \C P^{\infty} \ar[r]  & (\C P^{\infty})^k \ar[r]  & BU(k)}$$
In the diagram above the bottom left square and the bottom right squares are pullback squares of fibrations. Hence the composite 
\begin{equation}\label{comm1}
\xymatrix{ P_{\ell}W_{n,k} \ar[r] \ar[d]  & G_k(\C^n)\ar[d]\\
  \C P^{\infty} \ar[r]   & BU(k)}
\end{equation}
is also a pullback. These fibrations induce Serre spectral sequences 
\begin{equation}\label{ssgr}
E_2^{p,q} = H^p(BU(k))\otimes H^q(W_{n,k}) \implies H^{p+q}(G_k(\C^n))
\end{equation}
\begin{equation}\label{ssst}
E_2^{p,q} = H^p (\C P^\infty) \otimes H^q(W_{n,k}) \implies H^{p+q}(P_\ell W_{n,k})
\end{equation}
The pullback diagram (\ref{comm1}) induces a map between the two spectral sequences that commutes with the differentials. Recall that the cohomology of $BU(k)$ is a polynomial algebra on the classes $c_1,\cdots, c_k$ where $c_i=c_i(\xi_k)$, $\xi_k$ being the universal $k$-plane bundle. 

\begin{proposition}\label{diffgr}
In the spectral sequence (\ref{ssgr}) the classes $y_j\in H^*(W_{n,k})$ are transgressive and support the differential $d(y_j)=-c_j'$ (the classes $c_j'$ satisfy the equation $(1+c_1'+\cdots )(1+c_1+ \cdots + c_k)=1$).  
\end{proposition}
\begin{proof}
In the Serre spectral sequence for the fibration 
$$U(n)\rightarrow Gr_k(\mathbb{C}^n) \rightarrow B(U(k) \times U(n-k))$$ 
the differentials are given by $d(y_i) = c_i(\xi_{k} \oplus \xi_{n-k})$. We have the diagram of fibrations 
\begin{equation}\label{cdiffgr}
\xymatrix{U(n)\ar[r]^{q} \ar[d] & W_{n,k}\ar[d]\\
G_k(\mathbb{C}^n) \ar@{=}[r] \ar[d] & G_k(\mathbb{C}^n)\ar[d]\\
B(U(k) \times U(n-k)) \ar[r] & BU(k)}
\end{equation}
and hence a morphism of the associated Serre spectral sequences. Note that for $j\geq n-k+1$ the classes $y_j\in H^{2j-1}(W_{n,k})$ pullback to $y_j \in H^{2j+1}(U(n))$. We try to read off the expression for $d(y_j)$ in the spectral sequence for the right column from the one on the left. 

Denote by $E_*^{p,q}(l)$ the spectral sequence corresponding to the left vertical column in (\ref{cdiffgr}). The classes $y_1,\cdots,y_{n-k}$ are not in the image of $q^*$. Let $c_i=c_i(\xi_k)$ and $\tilde{c}_i=c_i(\xi_{n-k})$. The formula $d(y_j)=c_j(\xi_k\oplus \xi_{n-k})$ imply that in the page $E_{2(n-k)+1}(l)$, $\tilde{c}_i$ is equivalent to $c_i'$ for $i\leq n-k$. Hence 
$$d_{2(n-k)+2}(y_{n-k+1})=c_{n-k+1}(\xi_k \oplus \xi_{n-k}) =\sum_{i+j=n-k+1} c_ic_j'=-c_{n-k+1}'$$
in $E_{2(n-k)+2}(l)$. 

We observe that the equation $d_{2j}(y_j)= -c_j'$ holds for all $j\geq n-k+1$ in the page $E_{2j}(l)$. Proceeding by induction we have in the page $E_{2j-1}(l)$, $\tilde{c}_i=c_i'$ for $i\leq n-k$ and $c_i' = 0$ for $n-k+1\leq i\leq j-1$. Then in the page $E_{2j}(l)$ we have the equation
$$ \begin{array}{lll}
    d(y_{2j -1})
    &=& c_j(\xi_k\oplus \xi_{n-k}) \\
    &=& \displaystyle{\sum_{p+q=j}} c_p\tilde{c}_q\\
    &=& \displaystyle{\sum_{p+q=j,~p\leq k,~ q\leq n-k}}c_p c_q'\\
    &=& \displaystyle{\sum_{p+q=j,~p\leq k}} c_pc_q'\\
    &=& -c_j'\\
   \end{array}$$

Denote by $E_*^{p,q}(r)$ the spectral sequence for the right column of (\ref{cdiffgr}). For degree reasons the differentials $d_j$ are $0$ if $j<2(n-k)+2$. The morphism from the spectral sequence of the right column to the left column implies that the differentials on $y_j$ for $j>n-k$ are given by $d(y_i) =- c_i^{\prime}$.
\end{proof}

Next we translate the Proposition \ref{diffgr} to obtain differentials in the spectral sequence (\ref{ssst}). For a tuple $\ell=(l_1,\cdots,l_k)$ and integers $I=(i_1,\ldots,i_k)$ denote $|I|=\displaystyle{\sum_j} i_j$ and $\ell^I=\displaystyle{\prod_j} l_j^{i_j}$. We prove
\begin{proposition}\label{diffst}
In the spectral sequence (\ref{ssst}) the classes $y_j$ (for $j>n-k$) are transgressive and the differentials are given by $d(y_j)=-\displaystyle{\sum_{| I | = j}}  (-1)^j \ell^I x^j$.
\end{proposition} 

\begin{proof}
In the diagram (\ref{comm1}) the map $\phi_\ell: \C P^\infty \to BU(k)$ classifies the $k$-plane bundle $\xi^{l_1}\oplus \cdots \oplus \xi^{l_k}$. The chern classes of this bundle are computed by 
$$c(\oplus_j \xi^{l_j})=\prod_j (1+l_j x)$$
For the classes $c_j'$ define $c'=1+c_1'+\cdots$ so that $cc'=1$. This implies the pullback of $c'$ to $\C P^\infty$ is given by the equation
$$\phi_\ell^*c' = \prod_j (1+l_j x)^{-1}$$
Thus $\phi_\ell^*(c_j')= \displaystyle{\sum_{| I | = j}}  (-1)^j \ell^I x^j$. The result follows.
\end{proof}

Using the formulas above we may compute the cohomology of $P_\ell W_{n,k}$. 
\begin{theorem}\label{cohp}
 For an odd prime $p$, 
$$H^*(P_\ell W_{n,k};\Z/p)\cong  \Z/p[x]/(x^N) \otimes \Lambda (y_{n-k+1},\cdots,y_{N-1},y_{N+1},\cdots,y_n),$$
where $N= min_{r>n-k}\{\displaystyle{\sum_{| I | = r}}  \ell^I \neq 0 (mod~p)\}$.
\end{theorem}

\begin{proof}
We compute via the Serre spectral sequence (\ref{ssst}) with $\Z/p$ coefficients whose differentials are computed in Proposition \ref{diffst}. 

By the multiplicative structure, the first non-zero differential on a class in the vertical $0$-line is forced to be a transgression. With $N$ defined as in the statement note that the first non-zero transgression is given by $d_{2N}(y_N)=x^N$. Therefore the page $E_{2N+1}^{*,*}$ is isomorphic to the algebra 
$$\Z/p[x]/(x^N) \otimes \Lambda (y_{n-k+1},\cdots,y_{N-1},y_{N+1},\cdots,y_n).$$
Since the classes $y_j$ are transgressive there are no further differentials as $x^i=0$ for $i>N$ in $E_{2N+1}^{*,*}$. Hence $E_{2N+1}=E_\infty$. It follows that additively we must have 
$$H^*(P_\ell W_{n,k};\Z/p)\cong  \Z/p[x]/(x^N) \otimes \Lambda (y_{n-k+1},\cdots,y_{N-1},y_{N+1},\cdots,y_n)$$

To recover the multiplicative structure, observe that the factor $\Z/p[x]/(x^N)$ is a subalgebra as it comes from the horizontal $0$-line. Arbitrarily pick classes $y_j\in H^{2j-1}(P_\ell W_{n,k};\Z/p)$ (for $j>n-k$ and $j\neq N$) which pulls back to $y_j \in H^{2j-1}(W_{n,k};\Z/p)$ under the quotient map $W_{n,k}\to P_\ell W_{n,k}$. This exists by the additive computation above. The classes $y_j$ are odd dimensional classes and hence square to $0$. Thus multiplication induces a ring map 
$$ \Z/p[x]/(x^N) \otimes \Lambda (y_{n-k+1},\cdots,y_{N-1},y_{N+1},\cdots,y_n) \to H^*(P_\ell W_{n,k};\Z/p)$$
which is an additive isomorphism by the argument above. The result follows.
 \end{proof}

One may try to repeat the above argument for $p=2$ but then the squares on the classes $y_j$ might not be zero. However, if $k=2$ this case cannot arise and we have the following result
\begin{theorem}\label{coh2}
Let $\ell=(l_1,l_2)$ and suppose that  $2$ divides $\displaystyle{\sum_{p+q=n-1}} l_1^pl_2^q$. Then  
$$H^*(P_\ell W_{n,2};\Z/2)\cong  \Z/2[x,y_{n-1}]/(x^n, y_{n-1}^2).$$
Otherwise,
$$H^*(P_\ell W_{n,2};\Z/2)\cong  \Z/2[x,y_n]/(x^{n-1}, y_n^2).$$
\end{theorem}

\begin{proof}
The proof for Theorem \ref{cohp} may be repeated verbatim here. The only issue is with multiplicative extensions. Again choose representatives for $y_{n-1}, y_n$ in an arbitrary fashion. We examine the possible values for $y_j^2$. From dimension reasons no other class exists in the degree of $y_{n-1}^2$ and $y_n^2$ in either of the cases. The rest of the proof works as in Theorem \ref{cohp}. 
\end{proof}

\begin{exam}
Put $\ell=(1,\cdots,1)$ so that we recover the projective Steifel manifold. In that case note that $\displaystyle{\sum_{| I | = r}}  \ell^I $ is the number of ordered $k$-tuples of elements with sum $r$ which is $\binom{r+k-1}{k}=\binom{r+k-1}{r-1}$. 

Consider $N= \displaystyle{min_{r>n-k}}\{\binom{r+k-1}{r-1} \neq 0 (mod~p)\}$. The first term in this set is $r=n-k+1$ which is $\binom{n}{k}$. In view of the relation $\binom{r+k}{r}=\binom{r+k-1}{r-1}+\binom{r+k-1}{r}$, if $\binom{r+k-1}{r-1} \equiv 0 (mod~p)$, $ \binom{r+k}{r}=\binom{r+k-1}{r} (mod~p)$. Therefore one may rewrite the equation defining $N$ as $N= \displaystyle{min_{r>n-k}}\{\binom{n}{r} \neq 0 (mod~p)\}$. This matches with the cohomology computation in Theorem 1.1 of \cite{coh}.
\end{exam}

Refer to the commutative diagram (\ref{comm1}). This is actually a homotopy pullback. Hence one has an associated universal property for the manifold $P_\ell W_{n,k}$.
\begin{theorem}\label{univ}
The space $P_\ell W_{n,k}$ classifies line bundles $L$ for which there exists an $(n-k)$-bundle $E$ such that $E\oplus_j L^{l_j}$ is a trivial bundle. 
\end{theorem} 
\begin{proof}
Since the diagram (\ref{comm1}) is a homotopy pullback, $[X,P_\ell W_{n,k}]$ is equivalent to a map $X\to \C P^\infty$ and a map $X\to Gr_k(\C^n)$ such that the composites to $BU(k)$ are homotopic. Denote by $L$ the line bundle classified by the map to $\C P^\infty$ and by $E$ the pullback of the complementary canonical bundle $\xi_{n-k}$ over $Gr_k(\C^n)$. Then the maps are homotopic on composition to $BU(k)$ if and only if $\oplus_j L^{l_j} \oplus E = n\ep_\C$. The result follows. 
\end{proof}

\begin{rem}
If $\ell=(1,\cdots,1)$ the universal property classifies line bundles $L$ such that $kL\oplus E$ is a trivial bundle. We have the sequence of implications
$$kL\oplus E \cong n\ep_\C$$
$$\iff L^*\otimes E \otimes k \ep_\C \cong nL^*$$
$$\iff E^*\otimes L \oplus k\ep_\C \cong nL$$
Thus the universal property is equivalent to having $k$ linearly independent sections to the bundle $nL$. This reduces to the universal property in (5.2) of \cite{coh}.
\end{rem}

\section{Ranks of some complementary bundles}
For a vector bundle $\xi$, call a bundle $\eta$ complementary if $\xi \oplus \eta$ is a trivial bundle. In this section we use the computations of section \ref{coh} to deduce some results on the rank of complementary bundles on projective spaces and lens spaces. 

Recall that a complex vector bundle $\xi$ over a finite dimensional manifold $X$ always possesses a complementary bundle $\eta$ of dimension $\geq \frac{dim X}{2}$. In the following we use the spaces $P_\ell W_{n,2}$ to work out some examples concerning ranks of complementary bundles to $L^{l_1}\oplus L^{l_2}$ for fixed line bundles $L$ for $l_1$ and $l_2$ relatively prime. 

For the rest of the paper we fix the notation $\phi_d(l_1,l_2)= \frac{l_1^{d+1}-l_2^{d+1}}{l_1-l_2}$ so that for $\ell=(l_1,l_2)$, $\displaystyle{\sum_{|I|=d}} \ell^I = \phi_d(l_1,l_2)$.

\subsection{Complex Projective spaces}
We use the notation  $\xi$ to denote the canonical line bundle over $\C P^n$ and let $\zeta$ be such that $\xi^{l_1} \oplus \xi^{l_2}\oplus \zeta$ is a  trivial bundle. It follows that the total Chern class of $\zeta$ is $c(\zeta)= (1+l_1 x)^{-1}(1+l_2x)^{-1}$ so that  
$$c_i(\zeta)=x^i\sum_{p+q=i} (-1)^p l_1^p (-1)^ql_2^q =x^i(-1)^i \frac{l_1^{i+1} - l_2 ^{i+1}}{l_1-l_2}= \phi_i(l_1,l_2)x^i.$$ 
If $(l_1,l_2) \neq \pm (1,-1),$ or if $n$ is even, then $c_n(\zeta) \neq 0,$ hence $\mbox{rank}(\zeta)\geq n.$ 
Hence this gives an exact bound. We inspect in the case $l_1=1, l_2=-1$ and $n$ odd, whether there can exist $\zeta$ whose rank is $n-1$. In this case $c_n(\zeta)=0$.

From Theorem \ref{univ}, there exists complementary $\zeta$ of dimension $n-1$ if and only if there is a lift in the diagram 
\begin{equation}\label{lift}
\xymatrix{ & P_\ell W_{n+1,2} \ar[d] \\
\C P^n \ar[r]^{\xi}  & \C P^\infty }
\end{equation}
for $\ell=(l_1,l_2)$. We work our way up the Postnikov tower of $P_\ell W_{n+1,2}$. Observe from the fibre sequence $W_{n+1,2} \to P_\ell W_{n+1,2} \to \C P^\infty$ that the first non-trivial homotopy group of $P_\ell W_{n+1,2}$ is $\pi_2 \cong \Z$ and the first space in the Postnikov tower is $\C P^\infty$. The next stage for the Postnikov tower is
\begin{equation}\label{Post}
\xymatrix{
& P_2 \ar[d] &\\
P_{\ell}W_{n+1,2} \ar[ur]\ar[r] & \C P^{\infty} \ar[r] & K(\mathbb{Z},2n)}
\end{equation}
where $P_2 \to \C P^\infty \to K(\Z,2n)$ is a fibration. The map $P_\ell W_{n+1,2} \to P_2$ is $2n$-connected so that any possible obstruction to lift a map from $\C P^n$ must lie at this stage. Thus it suffices to consider the existence of the dotted arrow in 
 \begin{equation*}
\xymatrix{
& P_2 \ar[d] \\
\C P^n \ar@{-->}[ur]\ar[r]^{\xi} & \C P^{\infty}}
\end{equation*}
The lift exists if and only if the map $\C P^n \to \C P^\infty \to K(\Z,2n)$ is trivial. This is the restriction of the $k$-invariant $\alpha$ in $H^{2n}(\C P^\infty)$ to $\C P^n$.

 We have the commutative diagram of fibrations 
 \begin{equation}\label{D}
 \begin{CD}
W_{n+1,2} @>>> K(\mathbb{Z},2n-1) @>>>K(\mathbb{Z},2n-1)\\
 @VVV   @VVV @VVV\\
 P_{\ell}W_{n+1,2} @>>>P_2 @>>> PK(\mathbb{Z},2n)\\
 @VVV   @VVV @VVV\\
  \C P^{\infty} @=  \C P^{\infty} @>\alpha>>K(\mathbb{Z},2n)
\end{CD}
\end{equation}
In the top horizontal sequence the map $W_{n+1,2}\to K(\Z, 2n-1)$ is given by the generator $y_n \in H^{2n-1}(W_{n+1,2})$. Thus the class $\alpha$ is the transgression on the class $y_n$ which is given by $\phi_n(1,-1)=0$ as $n$ is odd. Thus the obstruction to the lift in the diagram (\ref{lift}) is $0$. Therefore we have proved the Proposition
\begin{proposition}\label{proj}
For $n$ odd, there is a $(n-1)$-bundle $\zeta$ over $\C P^n$ such that $\xi\oplus \xi^{-1} \oplus \zeta$ is trivial. ($\xi$ being the canonical line bundle) 
\end{proposition}
 In fact the same arguments as above also proves 
\begin{proposition}\label{proj2}
For any $\ell=(l_1,\ldots,l_k)$ of gcd $1$, there is a $(n-1)$-bundle $\zeta$ over $\C P^n$ such that $\oplus_j \xi^{l_j} \oplus \zeta$ is trivial if and only if $\displaystyle{\sum_{|I|=n}} \ell^I=0$.
\end{proposition}
 
 \subsection{Lens Space}
 Consider the Lens space $L^d(m) = \frac{S^{2d+1}}{\Z_m}$. There is  a quotient map $q_m : L^d(m) \rightarrow \C P^d$. Let $\lambda = q_m^*(\xi)$. Note that $\lambda^m$ is a trivial bundle. Consider the question: if  $\lambda^{l_1}\oplus\lambda^{l_2} \oplus \zeta $ is a trivial bundle, what are the possible restrictions on the rank of $\zeta$? 

The dimension of the lens space $L^d(m)$ equals $2d+1$, thus we can choose $\zeta$ to have dimension $d+1.$ Recall that the cohomology is given by ($u=c_1(\lambda)$)
$$H^*(L^d(m);\Z) \cong \Z[u,v_{2d+1}]/(mu,u^{d+1}, v_{2d+1}^2,uv_{2d+1})$$
We ask whether one can choose $\zeta$ of rank $d$. This is equivalent to the lift 
\begin{equation*}
\xymatrix{ & P_\ell W_{d+2,2} \ar[d] \\
L^d(m) \ar[r]^{\xi}  & \C P^\infty }
\end{equation*}
with $\ell=(l_1,l_2)$. Using the second stage of the tower (\ref{Post}) such an obstruction lies in $H^{2d+2}$ so that it is $0$ for $L^d(m)$. Therefore, the bundle $\zeta$ can be chosen to have rank $d$. The total chern class of $\zeta$ is
$$c(\zeta)=(1+l_1u)^{-1}(1+l_2u)^{-1}$$
This implies $c_d(\zeta) = \phi_d(l_1,l_2) \mbox{ mod } m$. Hence $\mbox{dim}(\zeta) \geq d$ if $m$ does not divide $\phi_d(l_1,l_2).$

For the next result we use some formulas for cohomology with $\Z/2$-coefficients. We note them now. The cohomology of $L^d(m)$ with $\Z/2$ coefficients is 
$$H^*(L^d(m),\Z_2)) \cong \Z[u,v]/(u^{d+1},v^2-\ep u)$$
with $\ep \equiv \frac{m}{2} (mod~2)$. The B\"ockstein homomorphism $\beta : H^1(L^d(m);\Z/2)\to H^2(L^d(m);\Z)$ is given by the formula $\beta(v)= \frac{m}{2}u$. Also we have the formula
$$Sq^2(u^{d-1}v)= (d-1)u^dv$$

Next consider $P_{\ell}W_{d+1,2}$. We have $H^{2d-1}(P_{\ell}W_{d+1,2};\Z/2) = \Z/2$ generated by $y_d,$ if $\phi_d(l_1,l_2) = 0 (mod~2).$ In this case the B\"ockstein 
$$\beta: H^{2d-1}(P_{\ell}W_{d+1,2};\Z/2)\cong \Z/2\{y_d\} \to H^{2d}(P_{\ell}W_{d+1,2};\Z) \cong \Z/{\phi(l_1,l_2)}\{x^d\}$$
is given by
$$\beta(y_d) = \frac{1}{2}\phi_d(l_1,l_2)x^d.$$
In this case also note that $y_{d+1}$ is $0$ in $H^{2d+1}(P_{\ell}W_{d+1,2};\Z/2)$ so that $Sq^2(y_d)=0$. Using these computations we prove 

 \begin{theorem}
 Suppose $d$ is even, $m$ is even, $m$ divides $\phi_d(l_1,l_2)$ and $\nu_2(m)=\nu_2(\phi_d(l_1,l_2)),$ where $\nu_2(n)$ denotes the two adic valuation of $n,$ then $\mbox{dim}(\zeta) \geq d.$
 \end{theorem} 
  
\begin{proof}
 Suppose $\mbox{dim}(\zeta) = d-1,$ then there exists $f : L^d(m) \rightarrow P_{\ell}W_{d+1,2}$ such that $f^*(x)=c_1(\lambda)=u$.  We have
$$\beta(y_d)=\frac{1}{2}\phi_d(l_1,l_2)x^d$$
$$\implies \beta(f^*(y_d))= \beta (\frac{1}{2}\phi_d(l_1,l_2)x^d)=\frac{\phi_d(l_1,l_2)}{2}u^d = \frac{m}{2} u^d$$
The last equality follows as $m$ divides $\phi_d(l_1,l_2)$ and $\nu_2(m)=\nu_2(\phi_d(l_1,l_2)).$
Note that $\beta: H^{2d-1}(L^d(m);\Z/2) \to H^{2d}(L^d(m);\Z)$ is injective and hence $f^*(y_d)=u^{d-1}v$.   Then,
$$
\xymatrix{ 
 H^{2d-1}( P_{\ell}W_{d+1,2};\Z/2) \ar[r]^{f^*} \ar[d]^{Sq^2} & H^{2d-1}(L^d(m);\Z/2)\ar[d]^{Sq^2} \\
H^{2d+1}( P_{\ell}W_{d+1,2};\Z/2) \ar[r]^{f^*}    & H^{2d+1}(L^d(m);\Z/2)}$$
implies $Sq^2f^*(y_d) = f^*Sq^2(y_d)$ but  $Sq^2f^*(y_d) = Sq^2u^{d-1}v = u^dv \neq 0$ as $d$ is even. However $f^*Sq^2(y_d) = f^*(0) =0,$ which leads to a contradiction. Hence it follows that $\mbox{dim}(\zeta) \geq d.$
\end{proof}

 \section{Applications}

In this section we consider applications of the cohomology computations in Section \ref{coh}. We write down the characteristic classes of the tangent bundle and the stable normal bundle of the manifold $P_\ell W_{n,2}$ and deduce bounds on the span and immersion codimension under some mild hypothesis.  

 The dimension of the manifold $P_\ell W_{n,2}$ is $4n-5$.  Note the expression for the tangent bundle for $P_\ell W_{n,2}$ from 2.2 of \cite{Gon-Su}. 
 $$  \tau(P_{\ell}W_{n,2}) \cong \xi^{-l_1}\otimes \xi^{l_2} \oplus \xi ^{-l_1} \otimes \beta \oplus \xi^{-l_2} \otimes \beta \oplus \ep_{\R} $$
In this expression $\xi$ is the line bundle associated to the principle $S^1$ bundle $ W_{n,2} \to P_{\ell}(W_{n,2})$ and $\beta$ satisfies $ \xi^{l_1} \oplus \xi^{l_2} \oplus \beta \cong \ep_{\R}^n.$ Eliminating the bundle $\beta$ from the equation above one has the following expression from Lemma 2.1 of \cite{Gon-Su}
 $$  \tau(P_{\ell}W_{n,2}) \oplus \xi^{-l_2}\otimes \xi^{l_1}\oplus 3\ep_{\R} \cong n ( \xi^{-l_1} \oplus \xi^{-l_2} ) $$
Observe that the first Chern class of the line bundle $\xi$ equals the class $x$ defined in section \ref{coh}. It follows that the total Pontrjagin class of the tangent bundle is given by (modulo $2$-torsion)
\begin{equation}\label{A1} 
 p(\tau(P_{\ell}W_{n,2}) ) = (1-l_1^2x^2)^n(1-l_2^2x^2)^n(1-(l_2-l_1)^2x^2)^{-1}
\end{equation}
Thus the Pontrjagin classes lie in the subalgebra of $H^*(P_\ell W_{n,2})$ generated by $x$. We have the following result on the span of these manifolds.

\begin{theorem}\label{span}
For any $\ell=(l_1,l_2)$ such that there is a prime $p$ dividing $n$ but not $l_2-l_1$, the span of $P_\ell W_{n,2}$ is $\leq 4n-5 - 2[\frac{n-2}{2}]$. In addition for $n$ odd  if $p$ divides $l_1^n - l_2^n$ then the span of $P_\ell W_{n,2}$ is $\leq 3n-4$.
\end{theorem}

\begin{proof}
We use that if the span of a vector bundle $\gamma$ is $k$ then the Pontrjagin classes $p_i(\gamma)$ are $0$ for $i > [\frac{\dim(\gamma)-k}{2}]$. In the spectral sequence for $H^*(P_\ell W_{n,2};\Z)$ of section \ref{coh} the first differential onto a power of $x$ possibly falls on the element $x^{n-1}$.  Therefore the powers $x^i$ are non-trivial for $i\leq n-2$.   

Suppose there is a prime $p$ which divides $n$ but not $l_2 - l_1$. Then the expression (\ref{A1}) modulo $p$ and $x^n$ is the same as $(1-(l_2-l_1)^2x^2)^{-1}$ which has non zero coefficient (modulo $p$) for every even power of $x$. Thus the coefficient of $x^{2[\frac{n-2}{2}]}$ is non-zero implying the first part.

 For $n$ odd write $n-1=2k$ and consider the possibility that $p_k(\tau(P_{\ell}W_{n,2})) \neq 0~ (mod~p)$. The expression  $(1-(l_2-l_1)^2x^2)^{-1}$  has a non-zero coefficient of $x^{n-1}$. Therefore $p_k(\tau(P_{\ell}W_{n,2})$ is non-zero if the class $x^{n-1}$ is non-zero in $H^*(P_\ell W_{n,2};\Z/p)$ which in turn is equivalent to the condition $\phi_{n-1}(l_1,l_2)=0 (mod~p)$. Note that $\phi_{n-1}(l_1,l_2)= \frac{l_1^n-l_2^n}{l_1-l_2}$. Hence the result follows. 
\end{proof}

\begin{rem}
Note that the second condition if easily satisfied (for example if $p-1$ divides $n$ and $p$ does not divide any $l_i$). There may be other similar results as the above. For example a similar argument demonstrates that if a prime $p$ divides $n-1$ but not $l_1-l_2$, $l_1$ or $l_1+3l_2$ the first conclusion holds. If in addition $p$ divides $l_1^n-l_2^n$ the second condition holds. One may make similar computations with $p$ dividing $n-2$ and so on. Thus it possible to write down many sets of divisibility relations for $l_1$ and $l_2$ which imply the first consequence, and in addition if the prime divides $l_1^n-l_2^n$ without dividing $l_1-l_2$ then the second consequence also follows.
\end{rem}
 
\mbox{ } \\

Next we consider the problem of immersing the manifold $P_\ell W_{n,2}$ in Euclidean space. If $P_{\ell}W_{n,2}$ is immersed in $\mathbb{R}^N$ for some $N$ then we have
$$\tau \oplus \nu \cong \ep_{\R}^N$$ 
where $\nu$ is the normal bundle. The total Pontrjagin classes modulo elements of order 2, satisfy $p(\nu) = p(\tau(P_{\ell}W_{n,2}))^{-1}.$  From (\ref{A1}) it follows that  
\begin{equation}\label{A2}
p(\nu) = (1-l_1^2x^2)^{-n}(1-l_2^2x^2)^{-n}(1-(l_2-l_1)^2x^2)
\end{equation}
We may prove the theorem
\begin{theorem}\label{imm}
Suppose that there exists a prime $p$ dividing $n-1$ and $l_2-l_1$. Then the class $p_{[\frac{n-3}{2}]}(\nu)$ is non-zero. Hence the manifold $P_lW_{n,2}$ does not immerse in $\R^{4n-5+2[\frac{n-3}{2}]}$.  
\end{theorem}

\begin{proof}
We compute $p(\nu)$ modulo $p$ and $x^n$ as in Theorem \ref{span}. Reducing (\ref{A2}) modulo $p$ and $x^n$ we get
$$p(\nu) = (1-l_1^2x^2)^{-2} (mod~p, x^n)$$
 The coefficient of $x^{2[\frac{n-3}{2}]}$ in this expression is $\binom{-2}{2[\frac{n-3}{2}]}=\pm (2[\frac{n-3}{2}]+1)$. This equals $n-2$ or $n-3$ none of which are divisible by $p$ as $p$ divides $n-1$.  
\end{proof}

\end{document}